\documentclass[%
aip,
amsmath,amssymb,
reprint, onecolumn 
]{revtex4-1}

\usepackage{graphicx}
\usepackage{dcolumn}
\usepackage{bm}

\usepackage[utf8]{inputenc}
\usepackage[T1]{fontenc}
\usepackage{mathptmx}
\usepackage{etoolbox}

\usepackage{amsmath, amssymb, amsfonts, amsthm, mathrsfs}
\newtheorem{theorem}{Theorem}

\newtheorem{example}{Example}%
\newtheorem{remark}{Remark}%
\newtheorem{lemma}{Lemma}%
\newtheorem{definition}{Definition}%

\makeatletter
\def\@email#1#2{%
	\endgroup
	\patchcmd{\titleblock@produce}
	{\frontmatter@RRAPformat}
	{\frontmatter@RRAPformat{\produce@RRAP{*#1\href{mailto:#2}{#2}}}\frontmatter@RRAPformat}
	{}{}
}%
\makeatother
\begin{document}
	
	\preprint{AIP}
	
	\title[Approximating the Symplectic Spectrum]{A Note on Approximating the Symplectic Spectrum}
	\author{V. B. Kiran Kumar}
	
	
	\affiliation{ 
		Department of Mathematics, Cochin University of Science and Technology, Kochi 682022
	}%
	
	\author{Anmary Tonny}%
	\email{anmarytonny97@gmail.com}
	\affiliation{%
		Department of Mathematics, Cochin University of Science and Technology, Kochi 682022
	}%
	
	\date{\today}
	
	\begin{abstract}    
		The symplectic eigenvalues play a significant role in finite mode quantum information theory, and Williamson's normal form proves to be a valuable tool in this area. Understanding the symplectic spectrum of a Gaussian Covariance Operator is a crucial task. Recently, in 2018, an infinite-dimensional analogue of Williamson’s Normal form was discovered, which has been instrumental in studying infinite mode Gaussian quantum states. However, most existing results pertain to finite-dimensional operators, leaving a dearth of literature in the infinite-dimensional context. The focus of this article is on employing approximation techniques to estimate the symplectic spectrum of certain infinite-dimensional operators. These techniques are well-suited for a particular class of operators, including specific types of infinite mode Gaussian Covariance Operators. Our approach involves computing the Williamson’s Normal form and deriving bounds for the symplectic spectrum of these operators. As a practical application, we explicitly compute the symplectic spectrum of Gaussian Covariance Operators. Through this research, we aim to contribute to the understanding of symplectic eigenvalues in the context of infinite-dimensional operators, opening new avenues for exploration in quantum information theory and related fields.
	\end{abstract}
	
	\maketitle

	\section{Introduction}

	In the finite mode quantum information theory,  the Williamson normal form is a mathematical tool used in the study of quantum systems, particularly in the field of quantum optics. It is named after John Williamson, who introduced it in $1936.$ The Williamson normal form provides a canonical representation of the covariance matrix of a Gaussian quantum state. A Gaussian quantum state is a type of quantum state that can be completely characterized by its first and second moments, which are related to its mean values and covariance matrix, respectively. The covariance matrix of a Gaussian quantum state describes the statistical correlations between pairs of observables in the state. The Williamson normal form allows us to diagonalize the covariance matrix, which means that it transforms it into a diagonal matrix. This diagonal form simplifies the analysis and calculation of properties of the quantum state. It reveals important information about the state, such as the squeezing properties and the uncertainties associated with different observables.	The symplectic eigenvalues of the covariance matrix provide important information about the quantum state under consideration. 	The connection between quantum information theory and symplectic eigenvalues arises through the study of covariance matrices associated with quantum states. The symplectic eigenvalues of these matrices provide valuable insights into the properties and capabilities of quantum systems, particularly those involving continuous variables.

	Infinite-mode quantum information theory provides tools and techniques for analyzing and manipulating quantum states in infinite-dimensional systems. It extends concepts such as quantum entanglement, quantum channels, and quantum measurements to the infinite-dimensional setting. This allows for a deeper understanding of the fundamental properties of quantum systems and opens up new possibilities for information processing and communication. In infinite-mode quantum information theory, one considers systems with infinite-dimensional Hilbert spaces. This framework is particularly relevant for studying quantum information processing in continuous-variable systems, where observables such as position and momentum are continuous rather than discrete. Examples of continuous-variable systems include quantum optics, where the state of light can be described using continuous variables such as the amplitude and phase of the electromagnetic field.

	Approximating the infinite-dimensional symplectic spectrum of a positive invertible operator can be a challenging task. However, there are some numerical methods and techniques that can be employed to obtain approximate solutions. One general approach is to express the positive invertible operator as a matrix or a set of matrices within the chosen finite-dimensional subspace,  use numerical techniques to compute the eigenvalues of the matrix representation, and once you have obtained the eigenvalues, analyze them to determine the approximate symplectic spectrum. It is important to note that the accuracy of the approximation will depend on the specific problem and the techniques used. Additionally, the choice of the finite-dimensional subspace and numerical methods will also affect the quality of the approximation.

	In this article, we consider this approximation problem of  symplectic spectrum for some special class of  positive invertible operators acting on infinite-dimensional Hilbert spaces. Recall the Williamson's normal form invented by  John Williamson in $1936.$ 
	
	\begin{theorem} (Williamson's normal form \cite{jw01})  \label{wnffd}
		Let $T \in M_{2n}(\mathbb{R})$ be a strictly positive matrix ($T$ is symmetric and has positive eigenvalues). Then there exist a symplectic matrix $L$ of order $2n \times 2n$ and a strictly positive diagonal matrix D of order $n \times n$ such that
		$$
		T = L^T
		\begin{bmatrix}
			D & 0 \\
			0 & D
		\end{bmatrix}L.
		$$
		Furthermore, D is unique upto the order of its entries.		
	\end{theorem}
	\noindent The entries of the diagonal matrix $D$ are defined as the symplectic eigenvalues of the matrix $T$, denote this set by $\sigma_{sy} (T).$ In $2018,$ B. V. Rajarama Bhat and T. C. John extended the Williamson's normal form to the infinite-dimensional settings; for strictly positive invertible operators on  real separable Hilbert spaces \cite{bv01}. The non-commutative analogues of the classical Gaussian distributions called Gaussian quantum states were studied in \cite{bv02}. The strictly positive invertible operators occur as covariance operators of Gaussian quantum states. Also, the symplectic spectrum (which we will define below) is a complete invariant for the class of covariance operators associated with the Gaussian quantum  states. Before stating the infinite-dimensional Williamson's normal form we give some definitions. 
	
	\begin{definition}\cite{bv01}
		Let $\mathcal{H}$ be a real Hilbert space and $I$ be the identity operator on $\mathcal{H}$. The involution operator $J$ on $\mathcal{H} \oplus \mathcal{H}$ is given by $ J = 
		\begin{bmatrix}
			0 & I \\
			-I & 0
		\end{bmatrix}.
		$ Let $\mathcal{H}$ and $\mathcal{K}$ be two real Hilbert spaces. A bounded invertible linear operator $Q : \mathcal{H} \oplus \mathcal{H} \rightarrow \mathcal{K} \oplus \mathcal{K}$ is called a symplectic transformation if $Q^{T}JQ = J$, where $J$ on the left side is the involution operator on $\mathcal{K} \oplus \mathcal{K}$ and that on the right side is the involution operator on $\mathcal{H} \oplus \mathcal{H}$.
	\end{definition}
	
	\begin{remark} \label{symcompo}
		Let $\mathcal{H,K,R}$ be real Hilbert spaces, $L: \mathcal{H} \oplus \mathcal{H} \rightarrow \mathcal{K} \oplus \mathcal{K}$ and $M: \mathcal{K} \oplus \mathcal{K} \rightarrow \mathcal{R} \oplus \mathcal{R}$ be symplectic transformations. Then $ML: \mathcal{H} \oplus \mathcal{H} \rightarrow \mathcal{R} \oplus \mathcal{R}$ is a symplectic transformation. This can be seen as follows. 
		
		\begin{align*}
			(ML)^{T}J(ML) &= (L^TM^T)J(ML) \\
			&= L^T(M^TJM)L \\
			&= L^TJL = J.	
		\end{align*}
	\end{remark}
	
	\noindent Below we state the infinite-dimensional Williamson's normal form for strictly positive invertible operators on infinite-dimensional separable real Hilbert spaces.
	\begin{theorem} (Williamson's normal form \cite{bv01}) \label{wnfifd}
		Let $\mathcal{H}$ be a real separable Hilbert space and $T$ be a strictly positive invertible operator on $ \mathcal{H} \oplus \mathcal{H}$ ($T$ is self-adjoint and has its spectrum on the positive real axis). Then there exists a positive invertible operator $D$ on $\mathcal{H}$ and a symplectic transformation $M : \mathcal{H} \oplus \mathcal{H} \rightarrow \mathcal{H} \oplus \mathcal{H}$ such that 
		$$
		T = M^T
		\begin{bmatrix}
			D & 0 \\
			0 & D
		\end{bmatrix}M$$ 
		Further, $D$ is unique upto conjugation with an orthogonal transformation. 
	\end{theorem}
	
	\noindent The spectrum of $D$ is called the symplectic spectrum of $T$, denoted by $\sigma_{sy}(T).$

	In this article, we consider the problem of approximating  $\sigma_{sy}(T)$ and its bounds on some special class of operators through the well-known truncation method. That means we consider the finite-dimensional trunctions $T_{2n}$ of $T$, which can be treated as even order matrices and compute $\sigma_{sy}(T_{2n})$. We analyze the limiting behaviour of these sets and see their connections with $\sigma_{sy}(T)$. To be more precise, we see that for the class of operators under consideration, the truncation method successfully approximate $\sigma_{sy}(T)$. Also, we are able to compute the bounds for $\sigma_{sy}(T)$. The Williamson's normal form is explicitly computed in such cases. Later, we will derive conditions under which the operators in the above-mentioned classes become a Gaussian Covariance Operator (GCO). Also, we will see that our approximation results are applicable to such operators. Indeed we show that the symplectic spectrum of a GCO can be completely retrieved using the truncation method given in \cite{bcn01}, and hence find useful in the infinite-mode quantum information theory. We  illustrate our results with some numerical examples at the end.

	The two classes of operators  considered in this article, are listed below. Let $\mathcal{H}$ be a real separable Hilbert space, $T$ be a positive invertible operator on $\mathcal{H} \oplus \mathcal{H}$. 
	\begin{enumerate}
		\item \underline{Class $\mathbb{A}$}: This class consists of operators $T$ of the form $T = \begin{bmatrix}
			A & 0 \\
			0 & B
		\end{bmatrix},$ where $A, B$ are positive invertible operators on $\mathcal{H}$ such that $AB = BA$.
		\item \underline{Class $\mathbb{B}$}: This class consists of operators $T$ of the form $T = \begin{bmatrix}
			A & B \\
			B & A
		\end{bmatrix},$ where $A$, $B$ are self-adjoint operators on $\mathcal{H}$ such that the operators $A + B$ and $A - B$ are positive invertible and commuting.
	\end{enumerate}

	\section{Preliminary Results} \label{mainresults}
	In this section we develop two major tools required to achieve our main results. The finite section method, also known as the truncation method, is a popular idea in the spectral approximation literature. It involves approximating an infinite-dimensional operator or matrix by considering its finite-dimensional submatrices. By truncating the operator or matrix to a finite size, we can work with a more manageable finite-dimensional problem. We will use only some important results from \cite{bcn01} here. 
	
	In the context of the infinite-dimensional Williamson normal form, the finite section method can be employed to approximate the covariance operator or the associated matrices by considering only a finite number of modes or truncating the system to a finite-dimensional subspace. This approximation allows for practical computations and analysis.
	
	The next idea is a reformulation of some classical complex Banach algebra results into the real Banach algebra settings. Many results in complex Banach algebra theory, such as the Gelfand-Mazur theorem, the spectral radius formula, or the holomorphic functional calculus, can be reformulated for real Banach algebras. The underlying ideas and proofs can often be directly translated from the complex setting to the real setting.  This enables us to utilize existing knowledge and techniques from complex Banach algebra theory, providing valuable insights and tools for studying real Banach algebras in the context of Gaussian quantum states.	The ideas here are straight forward translations from the complex case, and hence we avoid the proofs.

	\subsection{Truncation Method}
	Let $\mathcal{H}$ be a complex separable Hilbert space, $\{e_{1},e_{2},...\}$ be its countable orthonormal basis, and $A \in B(\mathcal{H})$ be self-adjoint. Define $P_{n}$ to be the orthogonal projection onto the subspace $\mathcal{H}_n =\textrm{ span }\{e_{1},...,e_{n}\}$. Also define $A_{n}$ to be the restriction of the operator $P_{n}AP_{n}$ to $\mathcal{H}_n$. That is $A_n$ can be considered as the truncation of $A$ and can be identified with an $n \times n$ matrix. The problem of spectral approximation includes the identification of the two limiting sets; \cite{bcn01}
	\begin{align*}
		\textrm{ lim inf }\sigma(A_n) &:= \{\lambda \in \mathbb{C} : \lambda \textrm{ is the limit of some sequence } (\lambda_{n})_{n=1}^{\infty} \textrm{ with } \lambda_{n} \in \sigma(A_{n}) \}, \\
		\textrm{ lim sup }\sigma(A_n) &:= \{\lambda \in \mathbb{C} : \lambda \textrm{ is the partial limit of some sequence } (\lambda_{n})_{n=1}^{\infty} \textrm{ with } \lambda_{n} \in \sigma(A_{n}) \}
	\end{align*} (in a similar way we can define lim inf $\sigma_{sy}(A_n)$ and lim sup $\sigma_{sy}(A_n)$).
	These sets need not be the same always.
	For example, (see \cite{bcn01}), consider the matrix $B(a,b) = \begin{bmatrix}
		a & b \\
		b & -a
	\end{bmatrix},$
	where $a,b \in \mathbb{R}$. Choose any sequence $(a_n)_{n = 1}^{\infty}$ of numbers $a_n \in (0,1)$ and define $b_n \in (0,1)$ by $a_n^2 + b_n^2 = 1,$ and  define $$A = \textrm{ diag }(B(a_1, b_1), B(a_2, b_2), \cdots).$$ Then $\sigma(A) = \{-1, 1\}$. Since $\sigma(A_{2n}) = \{-1, 1 \}$ and $\sigma(A_{2n+1}) = \{-1, a_n, 1\}$, we have lim inf $\sigma(A_n) = \{-1, 1\}$ while lim sup $\sigma(A_n)$ as the union of the set $\{-1,1\}$ and the set of all partial limits of the sequence $(a_n)$.

	Let $m = \textrm{ inf }\sigma(A)$, $M = \textrm{ sup }\sigma(A)$. Then we have the following inclusions 	for every selfadjoint operator $A \in B(\mathcal{H})$ \cite{bcn01};
	$$\{m, M\} \subset \sigma(A) \subset \textrm{ lim inf }\sigma(A_n) \subset \textrm{ lim sup } \sigma(A_n) \subset [m, M], $$

	\begin{definition} \cite{bcn01}
		Let $A$ be a bounded operator on $\mathcal{H}$. Then the essential spectrum of $A$ denoted as $\sigma_{ess}(A)$ is defined as follows;
		$$\sigma_{ess}(A) := \{\lambda \in \mathbb{C} : A - \lambda I + K(\mathcal{H}) \textrm{ is not invertible in } B(\mathcal{H}) / K(\mathcal{H})\},$$ where $K(\mathcal{H})$ is the set of all compact linear maps on $\mathcal{H}$. 
	\end{definition}
	If $A$ is self-adjoint, the essential spectrum consists of all spectral values which are not eigenvalues of finite multiplicity. The next theorem states that for operators whose essential spectrum is connected, the spectrum can be 
	approximated using finite dimensional truncations.
	
	\begin{theorem} \cite{bcn01} \label{bcnthm}
		Let $A \in B(\mathcal{H})$ be self-adjoint and suppose $\sigma_{ess}(A)$ is connected. Then $\textrm{ lim inf }\sigma(A_n) = \textrm{ lim sup } \sigma(A_n) = \sigma(A).$
	\end{theorem}

	\subsection{Results from Real Banach Algebras}\label{realb}
	Here we give some definitions and results from the real Banach algebras.
	\begin{definition}
		Let $\mathcal{B}$ be a real algebra, then the complexification of $\mathcal{B}$ is the complex algebra given by $\mathcal{A} = \mathcal{B} + i\mathcal{B}:= \{a + i.b : a,b \in \mathcal{B}\}$ with addition, multiplication and scalar multiplication defined respectively as follows:
		\begin{align*}
			(a + i.b) + (c + i.d) &= (a+b) + i.(c+d) \\
			(a+i.b)(c+i.d) &= (ac - bd) + i.(ad + bc) \\
			(\alpha + i. \beta)(a +i.b) &= (\alpha a - \beta b) + i.(\alpha b + \beta a)
		\end{align*} for all $a,b,c,d \in \mathcal{B}$ and $\alpha, \beta \in \mathbb{R}.$ 		
	\end{definition} 
	
	\noindent For a complex algebra $\mathcal{A}$ with unit $e$ and $a \in \mathcal{A}$, we define the spectrum of $a$ in $\mathcal{A}$ as the subset of $\mathbb{C}$ denoted as $\sigma(a)$ by
	$\sigma(a) = \{\lambda \in \mathbb{C}: a - \lambda e \textrm{ is not invertible in } \mathcal{A}\}.$ It is natural to think that for a real algebra $\mathcal{B}$ with unit $e$, the spectrum 
	of an element $a\in \mathcal{B}$  be defined as the set of all real $\lambda$, for which $a - \lambda e$, is not invertible in $\mathcal{B}$. But this definition will make the spectra of many elements to be the empty set. To avoid such situations, we adopt the following definition:
	
	\begin{definition} \cite{ik01}
		Let $\mathcal{B}$ be a real algebra with unit $e$. For $a \in \mathcal{B}$, the 
		spectrum of $a\in \mathcal{B}$ is the subset $\sigma(a)$ of $\mathbb{C}$ defined as $\sigma(a) := \{ s + it \in \mathbb{C} : (a - se)^2 + t^2e \textrm{ is not invertible in }\mathcal{B}\}.$ 
		Clearly, $s + it \in \sigma(a)$ if and only if $s - it \in \sigma(a)$.
	\end{definition}
	The above definition is equivalent to the following statement: "If $\mathcal{B}$ is a real algebra with unit and $\mathcal{A}$ is the complexification of $\mathcal{B}$ then $\sigma_{\mathcal{B}}(a) = \sigma_{\mathcal{A}}(a + i.0)$ for all $a \in \mathcal{B}$" \cite{shk01}. We know that for any complex Banach algebra $\mathcal{A}$, the spectrum of every $x \in \mathcal{A}$ is a compact subset of $\mathbb{C}$. This is true for real Banach algebras by Corollary $1.1.18$  \cite{shk01}.
	
	\begin{definition} \cite{shk01}
		Let $\mathcal{B}$ be a real algebra. The carrier space of $\mathcal{B}$, denoted by Car$(\mathcal{B})$, is the set of all nonzero homomorphisms from $\mathcal{B}$ to $\mathbb{C}$, regarded as a real algebra.
	\end{definition} 
	
	\noindent Let $\Phi \in$ Car$(\mathcal{B})$ and define $\overline{\Phi}$ by $\overline{\Phi}(a) = \overline{\Phi(a)}$ for all $a \in \mathcal{B}$. Then it is easy to see that $\overline{\Phi} \in $ Car$(\mathcal{B})$. Let $\tau: \textrm{ Car}(\mathcal{B}) \rightarrow \textrm{ Car}(\mathcal{B})$ be defined by $\tau(\Phi) = \overline{\Phi}$ for $\Phi \in $ Car$(\mathcal{B})$.
	
	\begin{definition} \cite{shk01}
		For $a$ in $\mathcal{B}$, the Gelfand transform of $a$ is the map $\hat{a} :$ Car$(\mathcal{B}) \rightarrow \mathbb{C}$, given by $\hat{a}(\Phi) = \Phi(a)$ for all $\Phi$ in Car$(\mathcal{B})$. The weakest topology on Car$(\mathcal{B})$ that makes $\hat{a}$ continuous on Car$(\mathcal{B})$ for all $a$ in $\mathcal{B}$ is called the Gelfand topology	on Car$(\mathcal{B})$.
	\end{definition} 
	
	\noindent Theorem $1.2.9(vi)$  \cite{shk01} states that if $\mathcal{B}$ is a real commutative Banach algebra with unit $e$ then, for every $a$ in $\mathcal{B}$, the range of $\hat{a}$ is the spectrum of $a$ in $\mathcal{B}$.
	
	\begin{definition} \cite{wr01}
		If $S$ is a subset of a real Banach algebra $\mathcal{B}$, the centralizer of $S$ is the set $\Gamma (S) = \{x \in \mathcal{B} : xs = sx \textrm{ for every } s \in S\}.$ We say that $S$ commutes if any two elements of $S$ commute with each other. 
	\end{definition} \noindent The next two theorems gives us the real analogue of Theorem 11.22 and Theorem 11.23 in \cite{wr01}. The proof is the same as in the complex case, hence we omit them.
	\begin{theorem}
		Suppose $\mathcal{B}$ is a real Banach algebra with unit $e, S \subset \mathcal{B}, S$ commutes and $\hat{\mathcal{B}} = \Gamma(\Gamma(S))$. Then $\hat{\mathcal{B}}$ is commutative real Banach algebra, $S \subset\hat{\mathcal{B}}$ and $\sigma_{\hat{\mathcal{B}}}(x) = \sigma_{\mathcal{B}}(x)$ for every $x \in \hat{\mathcal{B}}.$
	\end{theorem}
	
	\begin{theorem} \label{thminclu}
		Suppose $\mathcal{B}$ is a real commutative Banach algebra, $x,y \in \mathcal{B}$ and $xy = yx$. Then $$\sigma(x + y) \subset \sigma(x) + \sigma(y) \quad and \quad \sigma(xy) \subset \sigma(x)\sigma(y).$$
	\end{theorem}
	These spectral inclusion results for the product and sum of elements in a real Banach algebra provide valuable insights into the structure of the spectra and their relationship with the individual elements. They allow us to establish bounds on the spectra of the sum and product based on the spectra of the individual elements. We use them in the computation of bounds for symplectic spectrum of operators that lie in our classes, $\mathbb{A}$ and $\mathbb{B}$.
	
	\section{Approximation of Symplectic Spectrum}
	Throughout the rest of the article, $\mathcal{H}$ will denote a real separable Hilbert space. First we make a small observation in the finite-dimensional setting.
	
	\begin{lemma} \label{rstfdAA}
		Let $T \in M_{2n}(\mathbb{R})$ be a strictly positive matrix. Suppose $T$ takes the form $\begin{bmatrix}
			A & 0 \\
			0 & A
		\end{bmatrix}$, where $A$ is an $n \times n$ strictly positive matrix, then $\sigma_{sy}(T) = \sigma(A)$.
	\end{lemma}
	
	\begin{proof}
		Since $A$ is a strictly positive matrix, it is orthogonally diagonalizable, that is there exist an orthogonal matrix $O$ and a diagonal matrix $D$ such that $A = O^TDO.$ Also $\sigma(A) = \sigma(O^TDO)= \sigma(D).$ Define the matrix $L$ as $\begin{bmatrix}
			O & 0 \\
			0 & O
		\end{bmatrix}.$ Then it is easy to verify the identities,
		$	L^TJL  = J$ and $L^TL= LL^T = I$, that is $L$ is symplectic, and orthogonal. 
		
		Now  \begin{align*}
			L^T
			\begin{bmatrix}
				D & 0 \\
				0 & D
			\end{bmatrix}L &= 
			\begin{bmatrix}
				O^T & 0 \\
				0 & O^T
			\end{bmatrix}
			\begin{bmatrix}
				D & 0 \\
				0 & D
			\end{bmatrix}
			\begin{bmatrix}
				O & 0 \\
				0 & O
			\end{bmatrix} \\
			&= 
			\begin{bmatrix}
				O^TDO & 0 \\
				0 & O^TDO
			\end{bmatrix} \\ 
			&= 
			\begin{bmatrix}
				A & 0 \\
				0 & A
			\end{bmatrix} \\ &= T.
		\end{align*} That is, $\sigma_{sy}(T) = \sigma(D)$. Also since $L$ is orthogonal, the above set of equations gives $\sigma(D) = \sigma(T) = \sigma(A)$. Therefore, $\sigma_{sy}(T) = \sigma(D) = \sigma(A).$
	\end{proof}
	
	\begin{remark} \label{touseinthm}
		For the Williamson's normal form of matrices, the matrix $D$ needs to be diagonal while for the infinite-dimensional Williamson's normal form, $D$ is taken to be a positive invertible operator. Hence, in that case, the above lemma holds trivially with the identity transformation on $\mathcal{H} \oplus \mathcal{H}$ as the symplectic transformation.
	\end{remark}
	If $\{e_{1}, e_{2},...\}$ is the countable orthonormal basis for $\mathcal{H}$, then $\{(e_{1},0), (e_{2},0),...\}$ $\cup$ $\{(0, e_{1}), (0, e_{2}),...\}$ will be a countable orthonormal basis for $\mathcal{H} \oplus \mathcal{H}.$ Now define $P_{2n}$ as the orthogonal projection onto the $2n-$dimensional subspace $$(\mathcal{H} \oplus \mathcal{H})_{2n} = \textrm{ span } (\{(e_{1},0), (e_{2},0),...,(e_{n}, 0)\} \cup \{(0, e_{1}), (0, e_{2}),..., (0,e_{n})\}).$$ Define the operator $T_{2n}$ on $(\mathcal{H} \oplus \mathcal{H})_{2n}$ as the restriction of the operator $P_{2n}TP_{2n}$ to $(\mathcal{H} \oplus \mathcal{H})_{2n}$. With these preparatory notions, we can now state a positive result in symplectic spectral approximation for a very special class of operators.
	
	\begin{theorem} \label{rstAA}
		Let $A$ be a strictly positive invertible operator on $\mathcal{H}$. Consider the operator $T : \mathcal{H} \oplus \mathcal{H} \rightarrow \mathcal{H} \oplus \mathcal{H}$ of the form $T = 
		\begin{bmatrix}
			A & 0 \\
			0 & A
		\end{bmatrix}
		$. Suppose that the essential spectrum of $A$ is connected, then the symplectic spectrum can be approximated using the truncation method that is 
		$$\textrm{lim inf } \sigma_{sy}(T_{2n}) = \textrm{ lim sup } \sigma_{sy}(T_{2n}) = \sigma_{sy}(T).$$
	\end{theorem}
	
	\begin{proof}
		If $T$ takes the form $\begin{bmatrix}
			A & 0 \\
			0 & A
		\end{bmatrix}$, by Remark \ref{touseinthm} we have $\sigma_{sy}(T) = \sigma(A).$ 
		Also, $ T_{2n} = 
		\begin{bmatrix}
			A_{n} & 0 \\
			0 & A_{n}
		\end{bmatrix}. $ Hence by Lemma \ref{rstfdAA}, $\sigma_{sy}(T_{2n}) = \sigma(A_{n}),$ that is the symplectic spectral approximation of $T$ is actually the spectral approximation of $A$.	Since $\sigma_{ess}(A)$ is connected, by Theorem \ref{bcnthm} we have $\textrm{lim inf } \sigma(A_{n}) = \textrm{ lim sup } \sigma(A_{n}) = \sigma(A).$ That is, $$\textrm{lim inf } \sigma_{sy}(T_{2n}) = \textrm{ lim sup } \sigma_{sy}(T_{2n}) = \sigma_{sy}(T).$$
	\end{proof}
	
	\begin{example}
		Let $a$ be a positive real-valued continuous function in $L^{\infty}(\mathbb{T})$, where $\mathbb{T}$ is the complex unit circle, such that $a_{n} = a_{-n}$, for $n = 0,1,...$, where $a_{n}$s are the Fourier coefficients of $a$. Then the Toeplitz matrix
		$$ A = 
		\begin{bmatrix}
			a_{0} & a_{1} & a_{2} & \cdots \\
			a_{1} & a_{0} & a_{1} & \cdots \\
			a_{2} & a_{1} & a_{0} & \cdots \\
			\cdots & \cdots & \cdots & \cdots 
		\end{bmatrix}$$ induces a bounded, positive invertible operator on $l^{2}$. It is well-known that $\sigma(A) = \sigma_{ess}(A) = [\textrm{essinf } a, \textrm{ esssup } a],$ that is the essential spectrum of $A$ is connected. Therefore in this case with $T = 
		\begin{bmatrix}
			A & 0 \\
			0 & A
		\end{bmatrix}$ on $l^{2} \oplus l^{2}$, we have $\textrm{lim inf } \sigma_{sy}(T_{2n}) = \textrm{ lim sup } \sigma_{sy}(T_{2n}) = \sigma_{sy}(T) = \sigma(A).$
	\end{example} 
	
	The next example is motivated from Example 3.2 \cite{bcn01}. 
	
	\begin{example}
		Consider the function $f$ defined on $L^\infty[-\pi, \pi]$ as follows:
		$$
		f(x) = 
		\begin{cases}
			1, \quad x \in [-\pi, -\frac{\pi}{2}] \\
			2, \quad x \in [-\frac{\pi}{2}, \frac{\pi}{2}] \\
			1, \quad x \in [\frac{\pi}{2}, \pi]
		\end{cases}
		$$
		Then the $n^{th}$ Fourier coefficients $a_n$ is given as follows:
		$$
		a_n = a_{-n} =
		\begin{cases}
			\frac{3}{2}, \quad n = 0 \\
			\frac{1}{n\pi}, \quad n = 4k + 1 \\
			0, \quad n \textrm{ is non-zero and even} \\
			-\frac{1}{n\pi}, \quad n = 4k + 3
		\end{cases}
		$$
		and put 
		$$
		A = \left[ \begin{array}{c | cc| cc| cc|c}
			a_0 & a_{-1} & a_1 & a_{-2} & a_2 & a_{-3} & a_3 & \cdots \\
			\hline
			a_{1} & a_{0} & a_{2} & a_{-1} & a_{3} & a_{-2} & a_{4} & \cdots \\
			a_{-1} & a_{-2} & a_{0} & a_{-3} & a_{1} & a_{-4} & a_{2} & \cdots \\
			\hline
			a_{2} & a_{1} & a_{3} & a_{0} & a_{4} & a_{-1} & a_{5} & \cdots \\
			a_{-2} & a_{-3} & a_{-1} & a_{-4} & a_{0} & a_{-5} & a_{1} & \cdots \\
			\hline
			a_{3} & a_{2} & a_{4} & a_{1} & a_{5} & a_{-2} & a_{4} & \cdots \\
			a_{-3} & a_{-4} & a_{-2} & a_{-5} & a_{-1} & a_{-4} & a_{2} & \cdots \\
			\hline
			\cdots & \cdots & \cdots & \cdots & \cdots & \cdots & \cdots & \cdots
		\end{array}
		\right]
		$$ Then $A$ induces a bounded self-adjoint operator on $l^2$ such that $\sigma(A) = \sigma_{ess}(A) = \{ 1,2\}$ and lim sup $\sigma(A) = $ lim inf $\sigma(A) = [1,2]$. Now define $T : l^2 \oplus l^2 \rightarrow l^2 \oplus l^2$ as $T = \begin{bmatrix}
			A & 0 \\
			0 & A
		\end{bmatrix}$. Then $\sigma_{sy}(T) = \sigma(A) = \{1, 2\}$. The truncation method fails in this case as the limiting sets does not coincide with the spectrum of $A$.
	\end{example}

	\subsection{Williamson's normal form}
	
	Now we compute the Williamson's normal form of the two classes of operators defined in the introduction and approximate the symplectic spectrum of such operators.
	
	\begin{theorem} \label{rstAB}
		Let $A$, $B$ be strictly positive invertible operators on $\mathcal{H}$ such that $AB = BA$. Let $T : \mathcal{H} \oplus \mathcal{H} \rightarrow \mathcal{H} \oplus \mathcal{H}$ be defined as $$T = 
		\begin{bmatrix}
			A & 0 \\
			0 & B
		\end{bmatrix}.
		$$ Then $\sigma_{sy}(T) = \sigma(A^{\frac{1}{2}}B^{\frac{1}{2}}).$ In particular, if the essential spectrum of $A^{\frac{1}{2}}B^{\frac{1}{2}}$ is connected, then the symplectic spectrum can be approximated using the truncation method.
	\end{theorem}
	
	\begin{proof}
		Since $A$ and $B$ are positive invertible operators, they have unique square roots denoted by $A^{\frac{1}{2}}$ and $B^{\frac{1}{2}}$ respectively and they commute since $A$ and $B$ commute. Now define the operator $L$ on $\mathcal{H} \oplus \mathcal{H} \rightarrow \mathcal{H} \oplus \mathcal{H}$ as
		$$L = 
		\begin{bmatrix}
			A^{\frac{1}{4}}B^{- \frac{1}{4}} & 0 \\
			0 & A^{- \frac{1}{4}}B^{\frac{1}{4}}
		\end{bmatrix}
		$$
		Since $A$ and $B$ are positive and commuting so are their powers. Then, 
		
		\begin{align*}
			L^TJL &=
			\begin{bmatrix}
				A^{\frac{1}{4}}B^{- \frac{1}{4}} & 0 \\
				0 & A^{- \frac{1}{4}}B^{\frac{1}{4}}
			\end{bmatrix}^T
			\begin{bmatrix}
				0 & I \\
				-I & 0
			\end{bmatrix}
			\begin{bmatrix}
				A^{\frac{1}{4}}B^{- \frac{1}{4}} & 0 \\
				0 & A^{- \frac{1}{4}}B^{\frac{1}{4}}
			\end{bmatrix} \\
			&= \begin{bmatrix}
				(A^{\frac{1}{4}}B^{- \frac{1}{4}})^T & 0 \\
				0 & (A^{- \frac{1}{4}}B^{\frac{1}{4}})^T
			\end{bmatrix}
			\begin{bmatrix}
				0 & I \\
				-I & 0
			\end{bmatrix}
			\begin{bmatrix}
				A^{\frac{1}{4}}B^{- \frac{1}{4}} & 0 \\
				0 & A^{- \frac{1}{4}}B^{\frac{1}{4}}
			\end{bmatrix} \\
			&= \begin{bmatrix}
				A^{\frac{1}{4}}B^{- \frac{1}{4}} & 0 \\
				0 & A^{- \frac{1}{4}}B^{\frac{1}{4}}
			\end{bmatrix}
			\begin{bmatrix}
				0 & I \\
				-I & 0
			\end{bmatrix}
			\begin{bmatrix}
				A^{\frac{1}{4}}B^{- \frac{1}{4}} & 0 \\
				0 & A^{- \frac{1}{4}}B^{\frac{1}{4}}
			\end{bmatrix} \\
			&= \begin{bmatrix}
				0 & A^{\frac{1}{4}}B^{- \frac{1}{4}} \\
				-A^{- \frac{1}{4}}B^{\frac{1}{4}} & 0
			\end{bmatrix}
			\begin{bmatrix}
				A^{\frac{1}{4}}B^{- \frac{1}{4}} & 0 \\
				0 & A^{- \frac{1}{4}}B^{\frac{1}{4}}
			\end{bmatrix} \\
			&= \begin{bmatrix}
				0 & I \\
				-I & 0
			\end{bmatrix} \\
			&=J.
		\end{align*}
		
		That is $L$ is symplectic and 
		
		\begin{align*}
			L^T
			\begin{bmatrix}
				A^\frac{1}{2}B^{\frac{1}{2}} & 0 \\
				0 & A^\frac{1}{2}B^{\frac{1}{2}}		
			\end{bmatrix}L &= 
			\begin{bmatrix}
				A^{\frac{1}{4}}B^{- \frac{1}{4}} & 0 \\
				0 & A^{- \frac{1}{4}}B^{\frac{1}{4}}
			\end{bmatrix}^T
			\begin{bmatrix}
				A^\frac{1}{2}B^{\frac{1}{2}} & 0 \\
				0 & A^\frac{1}{2}B^{\frac{1}{2}}		
			\end{bmatrix}
			\begin{bmatrix}
				A^{\frac{1}{4}}B^{- \frac{1}{4}} & 0 \\
				0 & A^{- \frac{1}{4}}B^{\frac{1}{4}}
			\end{bmatrix} \\
			& = \begin{bmatrix}
				A^{\frac{1}{4}}B^{- \frac{1}{4}} & 0 \\
				0 & A^{- \frac{1}{4}}B^{\frac{1}{4}}
			\end{bmatrix}
			\begin{bmatrix}
				A^\frac{1}{2}B^{\frac{1}{2}} & 0 \\
				0 & A^\frac{1}{2}B^{\frac{1}{2}}		
			\end{bmatrix}
			\begin{bmatrix}
				A^{\frac{1}{4}}B^{- \frac{1}{4}} & 0 \\
				0 & A^{- \frac{1}{4}}B^{\frac{1}{4}}
			\end{bmatrix} \\
			&= \begin{bmatrix}
				A^{\frac{3}{4}}B^{\frac{1}{4}} & 0 \\
				0 & A^{\frac{1}{4}}B^{\frac{3}{4}}
			\end{bmatrix}
			\begin{bmatrix}
				A^{\frac{1}{4}}B^{- \frac{1}{4}} & 0 \\
				0 & A^{- \frac{1}{4}}B^{\frac{1}{4}}
			\end{bmatrix} \\
			&= \begin{bmatrix}
				A & 0 \\
				0 & B
			\end{bmatrix} \\
			&= T
		\end{align*}
		
		Since $A$ and $B$ are positive invertible and commuting, the operator $A^\frac{1}{2}B^{\frac{1}{2}}$ is positive invertible. Hence in this case we have $D = A^\frac{1}{2}B^{\frac{1}{2}}$ and $\sigma_{sy}(T) = \sigma(A^\frac{1}{2}B^{\frac{1}{2}})$.  Now suppose if $\sigma_{ess}(A^\frac{1}{2}B^{\frac{1}{2}})$ is connected, then by Theorem \ref{rstAA} we have
		$\sigma_{sy}(T) = \sigma(A^\frac{1}{2}B^{\frac{1}{2}}) = \textrm{lim inf } \sigma((A^\frac{1}{2}B^{\frac{1}{2}})_{n}) = \textrm{ lim sup } \sigma((A^\frac{1}{2}B^{\frac{1}{2}})_{n}).$
	\end{proof}
	We present some examples below to illustrate Theorem \ref{rstAB}. The examples of interest in quantum information theory, such as GCO, will appear later.
	\begin{example}
		Consider the real Hilbert space $l^2$. Let $A$ be the identity operator and define the diagonal operator $B$ as follows: 
		$$
		B = (b_{ij}) = \begin{cases}
			1 + \frac{1}{n}, \quad \textrm{ if } i = j = n^2 \\
			1, \qquad \textrm{       otherwise}
		\end{cases}
		$$ Define $T$ on $l^2 \oplus l^2$ as $T = \begin{bmatrix}
			A & 0 \\
			0 & B
		\end{bmatrix}$. Then $AB = BA$. The given operator satisfies the conditions in Theorem \ref{rstAB}. Then the operator $A^{\frac{1}{2}}B^{\frac{1}{2}}$ is the diagonal operator given by 
		$$
		A^{\frac{1}{2}}B^{\frac{1}{2}} = (c_{ij}) = \begin{cases}
			\left(1 + \frac{1}{n}\right)^{\frac{1}{2}}, \quad \textrm{ if } i = j = n^2 \\
			1, \qquad \textrm{       otherwise}
		\end{cases}
		$$ Then $\sigma_{sy}(T) = \sigma(A^{\frac{1}{2}}B^{\frac{1}{2}}) = \{1\} \cup \{ \left( 1 + \frac{1}{n} \right)^{\frac{1}{2}}, n \in \mathbb{N} \}$. Note that here $\sigma_{ess}(A^{\frac{1}{2}}B^{\frac{1}{2}}) = \{1\}$ which is connected and hence the symplectic spectrum can be approximated using the truncation method.
	\end{example}
	
	\begin{example}
		Define matrices $\tilde{A}$ and $\tilde{B}$ as follows.
		$$\tilde{A} = \begin{bmatrix}
			4 & 2 \\
			2 & 2
		\end{bmatrix}, \qquad \tilde{B} = \begin{bmatrix}
			2 & 1 \\
			1 & 1
		\end{bmatrix}.$$ Then $\tilde{A}$ and $\tilde{B}$ are positive definite matrices that commute. Now define operators $A_1$ and $B_1$ on the real Hilbert space $l^2$ as follows.
		\begin{align*}
			A_1 &= \tilde{A} \oplus \tilde{A} \oplus \cdots, \\
			B_1 &= \tilde{B} \oplus \tilde{B} \oplus \cdots.
		\end{align*} Then the operators $\tilde{A}$ and $\tilde{B}$ are positive invertible operators on $l^2$ that commute. Now define operators $A = A_1^2$ and $B = B_1^2$ on $l^2$, that is 
		\begin{align*}
			A = A_1^2 &= \tilde{A}^2 \oplus \tilde{A}^2 \oplus \cdots, \\
			B = B_1^2 &= \tilde{B}^2 \oplus \tilde{B}^2 \oplus \cdots.
		\end{align*} By commutativity and positivity of $\tilde{A}$ and $\tilde{B}$, we have the same for operators $A$ and $B$ respectively. Now define $T$ on $l^2 \oplus l^2$ as $$T = \begin{bmatrix}
			A & 0 \\
			0 & B
		\end{bmatrix}.$$ Then $T$ satisfies the conditions given in Theorem \ref{rstAB}. Also we have \begin{align*}
			A^{\frac{1}{2}}B^{\frac{1}{2}} &= A_1B_1 \\
			&= \begin{bmatrix}
				26 & 16 \\
				16 & 10
			\end{bmatrix} \oplus \begin{bmatrix}
				26 & 16 \\
				16 & 10
			\end{bmatrix} \oplus \cdots.
		\end{align*} Therefore $\sigma_{sy}(T) = \sigma(A^{\frac{1}{2}}B^{\frac{1}{2}}) = \{18 \pm 8\sqrt{5}\}$.
	\end{example}
	
	Recall that symplectic equivalence is a concept that arises in the field of symplectic geometry. We define symplectic equivalence as the relationship between two operators and observe that the symplectic spectrum remains invariant under this equivalence.
	
	\begin{definition}
		Let $\mathcal{K}$ be a real Hilbert space, $A$ and $B$ be operators on $\mathcal{K} \oplus \mathcal{K}.$ Then $A$ and $B$ are said to be \textbf{symplectically equivalent} if there exists a symplectic transformation $L : \mathcal{K} \oplus \mathcal{K} \rightarrow \mathcal{K} \oplus \mathcal{K}$ such that $A = L^{T}BL.$ 
	\end{definition}
	
	\begin{remark}
		Consider $\mathcal{H}, T, D$ as in Theorem \ref{wnfifd}. Then Theorem \ref{wnfifd} asserts that  the operators $T$ and $\begin{bmatrix}
			D & 0 \\
			0 & D
		\end{bmatrix}$ on $\mathcal{H} \oplus \mathcal{H}$ are symplectically equivalent through the symplectic transformation $M$.  
	\end{remark}
	
	The next theorem shows that positive invertible operators, which are symplectically equivalent, will have the same symplectic spectrum. This theorem plays an important role in the proof of Theorem \ref{rstABBA}.
	\begin{theorem} \label{rstsymeq}
		Let $A$, $B$ be strictly positive invertible operators on $\mathcal{H} \oplus \mathcal{H}$ such that $A$ and $B$ are symplectically equivalent. Then $\sigma_{sy}(A) = \sigma_{sy}(B).$
	\end{theorem}
	
	\begin{proof}
		Since the operators $A$ and $B$ are symplectically equivalent, there exist a symplectic transformation $L : \mathcal{H} \oplus \mathcal{H} \rightarrow \mathcal{H} \oplus \mathcal{H}$ such that $A = L^{T}BL.$ Since, $B$ is positive invertible on $\mathcal{H} \oplus \mathcal{H}$, by Theorem \ref{wnfifd}, there exists a symplectic transformation say $M$ and a positive invertible operator $D$ such that
		$$ B = M^{T}
		\begin{bmatrix}
			D & 0 \\
			0 & D
		\end{bmatrix}M,$$ and $\sigma_{sy}(B) = \sigma(D).$ Define $K = ML$. Then $K$ is a symplectic transformation as the composition of two symplectic transformations is again symplectic (by Remark \ref{symcompo}). Therefore, 
		$$A = L^{T}BL = L^{T}M^{T}
		\begin{bmatrix}
			D & 0 \\
			0 & D			
		\end{bmatrix}ML = K^{T}
		\begin{bmatrix}
			D & 0 \\
			0 & D			
		\end{bmatrix}K.$$ That is, $\sigma_{sy}(A) = \sigma(D) = \sigma_{sy}(B).$
	\end{proof}
	Now we state the final result of this section, which considers the class $\mathbb{B}$ operators.
	\begin{theorem} \label{rstABBA}
		Let $A$, $B$ be self-adjoint operators on $\mathcal{H}$ such that the operators $A + B$ and $A - B$ are positive invertible and commuting. Let $T : \mathcal{H} \oplus \mathcal{H} \rightarrow \mathcal{H} \oplus \mathcal{H}$ be defined as $$T = 
		\begin{bmatrix}
			A & B \\
			B & A
		\end{bmatrix}.$$ Then, $\sigma_{sy}(T) = \sigma((A + B)^{\frac{1}{2}}(A - B)^{\frac{1}{2}})$. In particular, if the essential spectrum of $(A + B)^{\frac{1}{2}}(A - B)^{\frac{1}{2}}$ is connected, then the symplectic spectrum can be approximated using the truncation method.
	\end{theorem}
	
	\begin{proof}
		Define a new operator $L : \mathcal{H} \oplus \mathcal{H} \rightarrow \mathcal{H} \oplus \mathcal{H}$ as $L = \frac{1}{\sqrt{2}}
		\begin{bmatrix}
			I & -I \\
			I & I
		\end{bmatrix}$, where $I$ is the identity operator on $\mathcal{H}$. Then 
		
		\begin{align*}
			L^TJL &= \left(\frac{1}{\sqrt{2}}
			\begin{bmatrix}
				I & -I \\
				I & I
			\end{bmatrix}\right)^T 
			\begin{bmatrix}
				0 & I \\
				-I & 0
			\end{bmatrix}\left(\frac{1}{\sqrt{2}}
			\begin{bmatrix}
				I & -I \\
				I & I
			\end{bmatrix}\right) \\
			&= \frac{1}{2}
			\begin{bmatrix}
				I & I \\
				-I & I
			\end{bmatrix}
			\begin{bmatrix}
				0 & I \\
				-I & 0
			\end{bmatrix}
			\begin{bmatrix}
				I & -I \\
				I & I
			\end{bmatrix} \\
			&= \frac{1}{2}
			\begin{bmatrix}
				-I & I \\
				-I & -I
			\end{bmatrix}
			\begin{bmatrix}
				I & -I \\
				I & I
			\end{bmatrix} \\
			&= \frac{1}{2}
			\begin{bmatrix}
				0 & 2I \\
				-2I & 0\\
			\end{bmatrix} 
			= J,
		\end{align*}
		that is $L$ is a symplectic transformation. 
		
		\begin{align*}
			\textrm{Also, }	T^\prime &= L^TTL \\
			&= \left(\frac{1}{\sqrt{2}}
			\begin{bmatrix}
				I & -I \\
				I & I
			\end{bmatrix}\right)^T
			\begin{bmatrix}
				A & B \\
				B & A 			\end{bmatrix}\left(\frac{1}{\sqrt{2}}
			\begin{bmatrix}
				I & -I \\
				I & I
			\end{bmatrix}\right) \\
			&= \frac{1}{2}
			\begin{bmatrix}
				I & I \\
				-I & I
			\end{bmatrix}
			\begin{bmatrix}
				A & B \\
				B & A
			\end{bmatrix}
			\begin{bmatrix}
				I & -I \\
				I & I
			\end{bmatrix} \\
			&= \frac{1}{2}
			\begin{bmatrix}
				A + B & A + B \\
				-A + B & A - B
			\end{bmatrix}
			\begin{bmatrix}
				I & -I \\
				I & I
			\end{bmatrix}\\
			&= \begin{bmatrix}
				A + B & 0 \\
				0 & A - B
			\end{bmatrix},
		\end{align*}
		
		that is $T$ and $T^{\prime}$ are symplectically equivalent. Hence, by Theorem \ref{rstsymeq}, $T$ and $T^{\prime}$ have the same symplectic spectrum. Hence by Theorem \ref{rstAB}, we have 
		$$\sigma_{sy}(T) = \sigma_{sy}(T^{\prime}) = \sigma((A + B)^{\frac{1}{2}}(A - B)^{\frac{1}{2}}).$$ Now suppose if $\sigma_{ess}((A + B)^{\frac{1}{2}}(A - B)^{\frac{1}{2}})$ is connected, then by Theorem \ref{rstAB} we have 
		$$\sigma_{sy}(T) = \textrm{ lim inf } \sigma(((A + B)^{\frac{1}{2}}(A - B)^{\frac{1}{2}})_n) = \textrm{ lim sup } \sigma(((A + B)^{\frac{1}{2}}(A - B)^{\frac{1}{2}})_n),$$ that is the symplectic spectrum can be approximated using finite dimesnional truncations.
	\end{proof}
	
	\begin{remark} \label{remarkonL}
		Note that the operator $L$ considered in the proof of Theorem \ref{rstABBA} is an orthognal transformation on $\mathcal{H} \oplus \mathcal{H}$. This can be seen as follows.
		\begin{align*}
			L^TL &= \left(\frac{1}{\sqrt{2}}
			\begin{bmatrix}
				I & -I \\
				I & I
			\end{bmatrix}\right)^T \left(\frac{1}{\sqrt{2}}
			\begin{bmatrix}
				I & -I \\
				I & I
			\end{bmatrix}\right) \\
			&= \frac{1}{2} \begin{bmatrix}
				I & I \\
				-I & I
			\end{bmatrix} \begin{bmatrix}
				I & -I \\
				I & I
			\end{bmatrix} \\
			&= \frac{1}{2} \begin{bmatrix}
				2I & 0 \\
				0 & 2I
			\end{bmatrix}  \\
			&= \begin{bmatrix}
				I & 0 \\
				0 & I
			\end{bmatrix} = I.
		\end{align*} Similarly we have $LL^T = I$. Hence $L$ is an orthogonal transformation on $\mathcal{H} \oplus \mathcal{H}$.
	\end{remark}
	
	\begin{example}
		Let $\mathcal{H} = l^2$. Let $A$ and $B$ be operators on $\mathcal{H}$ given with the matrix representations (with respect to some orthonormal basis)
		
		$$
		A = \begin{bmatrix}
			6 & 3 & \frac{1}{2} & 0 & 0 & 0 & 0 & \cdots \\
			3 & 6 & 3 & \frac{1}{2} & 0 & 0 & 0 & \cdots \\
			\frac{1}{2} & 3 & 6 & 3 & \frac{1}{2} & 0 & 0 & \cdots \\
			0 & \frac{1}{2} & 3 & 6 & 3 & \frac{1}{2} & 0 & \cdots \\
			\cdots & \cdots & \cdots & \cdots & \cdots & \cdots & \cdots & \cdots\\
		\end{bmatrix}
		$$
		$$
		B = \begin{bmatrix}
			5 & 3 & \frac{1}{2} & 0 & 0 & 0 & 0 & \cdots \\
			3 & 5 & 3 & \frac{1}{2} & 0 & 0 & 0 & \cdots \\
			\frac{1}{2} & 3 & 5 & 3 & \frac{1}{2} & 0 & 0 & \cdots \\
			0 & \frac{1}{2} & 3 & 5 & 3 & \frac{1}{2} & 0 & \cdots \\
			\cdots & \cdots & \cdots & \cdots & \cdots & \cdots & \cdots & \cdots\\
		\end{bmatrix}
		$$
		That is $A$ is the Toeplitz operator corresponding to the symbol $6 + 6cos(t) + cos (2t) \in L^\infty(\mathbb{T})$ and $B$ is the Toeplitz operator corresponding to the symbol $5 + 6cos(t) + cos (2t) \in L^\infty(\mathbb{T})$. So $A$ and $B$ are self-adjoint. Also, 
		$$
		A + B = \begin{bmatrix}
			11 & 6 & 1 & 0 & 0 & 0 & 0 & \cdots \\
			6 & 11 & 6 & 1 & 0 & 0 & 0 & \cdots \\
			1 & 6 & 11 & 6 & 1 & 0 & 0 & \cdots \\
			0 & 1 & 6 & 11 & 6 & 1 & 0 & \cdots \\
			\cdots & \cdots & \cdots & \cdots & \cdots & \cdots & \cdots & \cdots\\
		\end{bmatrix}
		$$ and $A - B$ is the identity operator. Hence $A + B$ and $A - B$ commutes. Therefore $A$ and $B$ satisfies the conditions in Theorem \ref{rstABBA}. Now define $T: l^2 \oplus l^2 \rightarrow l^2 \oplus l^2$ as $T = \begin{bmatrix}
			A & B \\
			B & A 
		\end{bmatrix}$. From Theorem \ref{rstABBA}, $$\sigma_{sy}(T) = \sigma((A + B)^{\frac{1}{2}}(A - B)^{\frac{1}{2}}) = \sigma((A + B)^{\frac{1}{2}}) = \sigma(A + B)^{\frac{1}{2}}.$$ $A + B$ being a Toeplitz operator corresponding to the symbol $f(t) = 11 + 12 cos(t) + 2 cos(2t)$, $\sigma(A + B)$ is the essential range of $f$, which is the interval $[1,25]$. Therefore, $\sigma_{sy}(T) = \sigma(A + B)^{\frac{1}{2}} = [1,5].$ Also since $\sigma_{ess}((A + B)^{\frac{1}{2}}(A - B)^{\frac{1}{2}}) = [1,5]$ (being a Toeplitz operator) is connected, the symplectic spectrum can be approximated using the truncation method.
	\end{example}
	
	\begin{example}
		Consider the positive invertible operator $T$ on $\mathcal{H} \oplus \mathcal{H}$ defined as
		$$T = \begin{bmatrix}
			A & B \\
			B & A
		\end{bmatrix},$$ where $A$ and $B$ are self-adjoint operators given by \begin{align*}
			A &= \textrm{ diag }\left\{2 + \frac{1}{n^2} + \frac{1}{n^3}: n \in \mathbb{N}\right\}, \\ 
			B &= \textrm{ diag }\left\{\frac{1}{n^2} - \frac{1}{n^3}: n \in \mathbb{N}\right\}.
		\end{align*} Now $A + B = \left\{ 2 + \frac{2}{n^2}: n \in \mathbb{N}\right\}$ and $A - B = \left\{ 2 + \frac{2}{n^3}: n \in \mathbb{N} \right\}$, that is $A + B$ and $A - B$ are positive invertible operators and being diagonal they commutes. Hence $A$ and $B$ satisfies the conditions given in Theorem \ref{rstABBA}. Now
		$$T^\prime = 
		\begin{bmatrix}
			A-B & 0 \\
			0 & A+B
		\end{bmatrix}. $$ Hence by Theorem \ref{rstABBA} \begin{align*}
			\sigma_{sy}(T) = \sigma_{sy}(T^\prime) &= \sigma(((A + B)^{\frac{1}{2}}(A - B)^{\frac{1}{2}})) \\
			&= \{2\} \cup \left\{ \left[ \left( 2 + \frac{2}{n^2}\right)\left( 2 + \frac{2}{n^3}\right)\right]^{\frac{1}{2}} : n \in \mathbb{N} \right\}.
		\end{align*} 
		
	\end{example}
	
	In Section \ref{secnumillus}, we give a numerical example to illustrate the truncation method discussed above.

	\section{Bounds for the Symplectic Spectrum}
	A consequence of Theorem $11$\cite{rbt01} is that the symplectic eigenvalues of a positive definite real matrix of even order lie between the smallest and the largest eigenvalues. In this section, we prove the same result for the two classes of operators considered in this article. We demonstrate that for the operators in Class $\mathbb{A}$ and Class $\mathbb{B}$, the symplectic spectrum falls within the bounds of the spectrum. The main tools used to accomplish this are the results from the real Banach algebras discussed in Section  \ref{realb}.

	\begin{theorem} \label{rstinclu}
		Let $T$ be an operator in Class $\mathbb{A}$ or Class $\mathbb{B}$. Then the symplectic spectral values of $T$ lies in $[m, M]$, where $m = \textrm{ inf }\sigma(T)$, $M = \textrm{ sup }\sigma(T)$. 
	\end{theorem}
	
	\begin{proof}
		\noindent \underline{\textbf{Case I:}} $T$ is in Class $\mathbb{A}$.
		
		\noindent Here $\sigma(T) = \sigma(A) \cup \sigma(B).$ Since $T$ is positive, $m, M >0$. We have seen that $\sigma_{sy}(T) = \sigma(A^{\frac{1}{2}}B^{\frac{1}{2}})$. 
		
		\noindent Since, $\sigma(T) \subset [m,M]$, we have $ \sigma(A) \subset [m,M] \Rightarrow \sigma(A^{\frac{1}{2}}) = \sigma(A)^{\frac{1}{2}} \subset [m^{\frac{1}{2}}, M^{\frac{1}{2}}].$ Similarly, $\sigma(B^{\frac{1}{2}}) = \sigma(B)^{\frac{1}{2}} \subset [m^{\frac{1}{2}}, M^{\frac{1}{2}}].$
		
		\noindent From Theorem \ref{thminclu} we have, $\sigma_{sy}(T) = \sigma(A^{\frac{1}{2}}B^{\frac{1}{2}}) \subset \sigma(A^{\frac{1}{2}}) \sigma(B^{\frac{1}{2}}) \subset [m, M],$ that is the symplectic spectral values lies in $[m, M]$.  
		
		\noindent \underline{\textbf{Case II:}} $T$ is in Class $\mathbb{B}$.
		
		\noindent When this case was considered, we showed that $T$ is symplectically equivalent to the operator given by $T^{\prime} = \begin{bmatrix}
			A + B & 0 \\
			0 & A - B
		\end{bmatrix},$ by the symplectic and orthogonal transformation $L = \frac{1}{\sqrt{2}}\begin{bmatrix}
			I & -I \\
			I & I		
		\end{bmatrix}$.  Hence $\sigma(T) = \sigma(T^{\prime}).$ Now proceeding as in Case II, we have
		$$\sigma_{sy}(T) = \sigma((A + B)^{\frac{1}{2}}(A - B)^{\frac{1}{2}}) \subset \sigma((A + B)^{\frac{1}{2}}) \sigma((A - B)^{\frac{1}{2}}) \subset [m,M].$$
	\end{proof}
	
	\begin{example} \label{examplesyminclu}
		Let $H = l^2$, $A$ and $B$ be positive invertible operators on $l^2$ defined as follows: let $a = 2 + cos(t) \in L^\infty(\mathbb{T}) $. Then the Toeplitz matrix
		$$ A = 
		\begin{bmatrix}
			2 & \frac{1}{2} & 0 & 0 & \cdots \\
			\frac{1}{2} & 2 & \frac{1}{2} & 0 & \cdots \\
			0 & \frac{1}{2} & 2 & \frac{1}{2} & \cdots \\
			\cdots & \cdots & \cdots & \cdots & \cdots 
		\end{bmatrix}
		$$ will define a positive invertible operator on $l^2$. Similarly define the operator $B$ on $l^2$ by the Toeplitz matrix  
		$$ B = 
		\begin{bmatrix}
			2 & -\frac{1}{2} & 0 & 0 & \cdots \\
			-\frac{1}{2} & 2 & -\frac{1}{2} & 0 & \cdots \\
			0 & -\frac{1}{2} & 2 & -\frac{1}{2} & \cdots \\
			\cdots & \cdots & \cdots & \cdots & \cdots 
		\end{bmatrix}$$ which corresponds to the symbol $b = 2 - cos(t) \in L^\infty(\mathbb{T})$. Then, $\sigma(A) = \sigma(B) = [1,3]$. Now define $T: l^2 \oplus l^2 \rightarrow l^2 \oplus l^2$ by $T = \begin{bmatrix}
			A & 0 \\
			0 & B
		\end{bmatrix}$. Then $\sigma(T) = \sigma(A) \cup \sigma(B) = [1,3].$ Hence by Theorem \ref{rstinclu}, any real $r \in \mathbb{R} \setminus [1,3]$ will not be a symplectic spectral value of the operator $T$.		
	\end{example}

	\section{Application to Gaussian Covariance Operators} \label{sectionapplicationstogco}
	
	In this section we derive the conditions for operators in Class $\mathbb{A}$ and Class $\mathbb{B}$ to be a  Gaussian Covariance Operator (GCO). Later we shall show that the symplectic spectrum of a GCO can be completely retrieved using the truncation method described in \cite{bcn01}. Recall that the Gaussian Covariance Operators are the covariance operators associated with infinte mode Gaussian quantum states. A characterisation for these operators given in \cite{bv02} will be used in this article.
	
	First let us recall the definitions of complexification of a real Hilbert space and complexification of an operator.
	
	\begin{definition} \label{defncomplexification} \cite{bv01}
		By the complexification of a real Hilbert space $\mathcal{K}$ we mean the complex Hilbert space $\hat{\mathcal{K}} = \mathcal{K} + i \cdot \mathcal{K} = \{x + i \cdot y : x,y \in \mathcal{K} \}$ with addition, complex-scalar product and inner product defined in the obvious way. For a bounded operator $T$ on the real Hilbert space $\mathcal{K}$, define an operator $\hat{T}$ on the complexification $\hat{\mathcal{K}}$ of $\mathcal{K}$ by $\hat{T}(x + i \cdot y) = Tx + i \cdot Ty$. $\hat{T}$ is called the complexification of $T$. Define the spectrum of $T$, denoted by $\sigma(T)$, to be the spectrum of $\hat{T}$, that is the spectrum of an operator and its complexification is the same.
	\end{definition}
	
	\begin{remark} \label{remarkcomplexificationprop}
		We make the following observations on the complexification of operators.
		\begin{enumerate}
			\item For operators $T_1$ and $T_2$ on $\mathcal{K}$, $\widehat{T_1T_2} = \hat{T_1}\hat{T_2}$. This can be seen as follows. Let $x = a + ib \in \hat{\mathcal{K}}$, where $a,b \in \mathcal{K}$. Then 
			\begin{align*}
				\widehat{T_1T_2}(x) &= \widehat{T_1T_2}(a+ib) = (T_1T_2)(a) + i(T_1T_2)(b) \\
				&= T_1(T_2(a)) + iT_1(T_2(b)) \\
				&= \hat{T_1}(T_2(a) +iT_2(b)) \\
				&= \hat{T_1}(\hat{T_2}(a + ib)) = \hat{T_1}(\hat{T_2}(x)), \quad \forall x \in \hat{\mathcal{K}}.
			\end{align*} Therefore $\widehat{T_1T_2} = \hat{T_1}\hat{T_2}.$
			
			\item For an operator $T$ on $\mathcal{K}$, $(\hat{T})^* = \widehat{(T^T)}.$ This follows from Definition 2.3\cite{bv01}. Hence $\hat{T}$ is Hermitian if and only if $T$ is self-adjoint.
		\end{enumerate}
	\end{remark}
	
	Now we define Gaussian Covariance Operators.
	\begin{definition} [Gaussian Covariance Operators \cite{bv02}]
		\label{defgco}
		Let $S$ be a real linear, bounded, symmetric and invertible operator on $\mathcal{H} \oplus \mathcal{H}$. Then $S$ is called a Gaussian Covariance Operator (GCO) if the following three conditions are satisfied.
		\begin{enumerate}
			\item $\hat{S} - i\hat{J} \geq 0$, where $\hat{S}$ and $\hat{J}$ are the complexification of the operators $S$ and $J$ respectively, 
			\item $S - I$ is Hilbert-Schmidt,
			\item $(\sqrt{S}J\sqrt{S})^T(\sqrt{S}J\sqrt{S}) - I$ is of trace class.
		\end{enumerate}		
	\end{definition}
	
	\begin{remark} \label{remarkcond3}
		The third condition in Definition \ref{defgco} can be replaced by the condition $(JS)^2 + I$ is of trace class \cite[Corollary 3.3.1]{john2018infinite}.
	\end{remark}

	\subsection{Conditions for operators in Class $\mathbb{A}$ and Class $\mathbb{B}$ to be a GCO}\label{subseccondGCO}
	
	In this section, we shall formulate the conditions under which the operators in Class $\mathbb{A}$ and Class $\mathbb{B}$ becomes a GCO.
	
	\begin{theorem} \label{thmgcoinclassA}
		Let $S$ be an operator in Class $\mathbb{A}$. Then $S$ is a GCO if and only if
		\begin{enumerate}
			\item $\hat{A} \geq \hat{B}^{-1}$ on $\hat{\mathcal{H}}$,
			\item $A - I$ and $B - I$ are Hilbert-Schmidt operators on $\mathcal{H}$,
			\item $AB - I$ is a trace class operator on $\mathcal{H}$.
		\end{enumerate}
	\end{theorem}
	
	\begin{proof}
		Given that the operator $S$ is in Class $\mathbb{A}$. Then $S$ takes the form 
		$S = \begin{bmatrix}
			A & 0 \\
			0 & B
		\end{bmatrix},$ where $A$ and $B$ are positive invertible operators on $\mathcal{H}$ such that $AB = BA.$ Now for $S$ to be a GCO, it should satisfy the conditions given in Definition \ref{defgco}. Let us see each of them separately. The first condition says that $\hat{S} - i\hat{J} \geq 0$. That is,
		\begin{align*}
			\hat{S} - i \hat{J} \geq 0 &\iff \begin{bmatrix}
				\hat{A} & 0 \\
				0 & \hat{B}
			\end{bmatrix} - i \begin{bmatrix}
				0 & \hat{I} \\
				-\hat{I} & 0
			\end{bmatrix} \geq 0 \\
			&\iff \begin{bmatrix}
				\hat{A} & 0 \\
				0 & \hat{B}
			\end{bmatrix} - i \begin{bmatrix}
				0 & I \\
				-I & 0
			\end{bmatrix} \geq 0 \\
			&\iff \begin{bmatrix}
				\hat{A} & -iI \\
				iI & \hat{B} 
			\end{bmatrix} \geq 0
		\end{align*} From Theorem 1.3.3\cite{bhatia2009positive}, we have
		$\begin{bmatrix}
			\hat{A} & -iI \\
			iI & \hat{B} 
		\end{bmatrix} \geq 0 \iff \hat{A} \geq (-iI)(\hat{B}^{-1})(iI).$
		Hence 
		$
		\hat{S} - i\hat{J} \geq 0 \textrm{ on } \widehat{\mathcal{H} \oplus \mathcal{H}} \iff \hat{A} \geq \hat{B}^{-1} \textrm{ on } \hat{\mathcal{H}}.
		$ \newline
		The second condition says that $S - I$ is Hilbert-Schmidt. 
		This happens when $S - I =
		\begin{bmatrix}
			A & 0 \\
			0 & B
		\end{bmatrix} - I = \begin{bmatrix}
			A - I & 0 \\
			0 & B - I
		\end{bmatrix}$ is Hilbert-Schmidt, that is, when $A - I$ and $B - I$ are Hilbert-Schmidt operators on $\mathcal{H}$. Hence the second condition holds if and only if 
		$A - I$ and $B - I$ are Hilbert-Schmidt operators on $\mathcal{H}$.	\\
		The third condition says that $(JS)^2 + I$ is of trace class. Now $$JS = \begin{bmatrix}
			0 & I \\
			-I & 0
		\end{bmatrix} \begin{bmatrix}
			A & 0 \\
			0 & B
		\end{bmatrix} = \begin{bmatrix}
			0 & B \\
			-A & 0
		\end{bmatrix}.$$
		$$\Rightarrow (JS)^2 = \begin{bmatrix}
			-BA & 0 \\
			0 & -AB
		\end{bmatrix} = \begin{bmatrix}
			-AB & 0 \\
			0 & -AB
		\end{bmatrix}.$$
		Then $(JS)^2 + I$ is of trace class if and only if
		$\begin{bmatrix}
			-AB & 0 \\
			0 & -AB
		\end{bmatrix} + I = \begin{bmatrix}
			-AB + I & 0 \\
			0 & -AB +I
		\end{bmatrix}$ is of trace class. That is, 
		$AB - I$ is of trace class on $\mathcal{H}$.	\end{proof}
	
	Before going to the next theorem we show that if two operators $S$ and $R$ are symplectically and orthogonally equivalent, then $S$ is a GCO if and only if $R$ is a GCO.
	
	\begin{lemma} \label{lemmaconditionongco}
		Let $S$ and $R$ be two operators on $\mathcal{H} \oplus \mathcal{H}$ such that $R = M^TSM$ for some symplectic and orthogonal operator $M$ on $\mathcal{H} \oplus \mathcal{H}$. Then $S$ is a GCO if and only if $R$ is a GCO.
	\end{lemma}
	
	\begin{proof}
		Given that the operators $S$ and $R$ are such that $R = M^TSM$ for some symplectic and orthogonal operator $M$ on $\mathcal{H} \oplus \mathcal{H}$.
		Suppose $S$ is a GCO. We claim that $R$ is a GCO, that is $R$ satisfies the three conditions in Definition \ref{defgco}.
		\begin{enumerate}
			\item To show $\hat{R} - i \hat{J} \geq 0$:
			\begin{align*}
				\hat{R} - i \hat{J} &= \widehat{M^TSM} - i \widehat{M^TJM} \\
				&= \widehat{M^T}\hat{S}\hat{M} - i \widehat{M^T}\hat{J}\hat{M} \quad (\textrm{follows from Remark \ref{remarkcomplexificationprop}})\\
				&= \widehat{M^T}(\hat{S} - i \hat{J})\hat{M}.
			\end{align*}
			Since $(\widehat{M^T})^* = ((\hat{M})^*)^* = \hat{M}$ (from Remark \ref{remarkcomplexificationprop}) and $\hat{S} - i\hat{J} \geq 0$, we have 
			$\hat{R} - i\hat{J} = \widehat{M^T}(\hat{S} - i\hat{J})\hat{M} \geq 0.$
			
			\item To show $R - I$ is Hilbert-Schmidt:
			\begin{align*}
				R - I &= M^TSM - I \\
				&= M^TSM - M^TM \\
				&= M^T(S - I)M.
			\end{align*}
			Since $S - I$ is Hilbert-Schmidt and they form a two sided ideal, we have the operator $R - I$ to be Hilbert-Schmidt.
			
			\item To show $(JR)^2 + I$ is of trace class.
			\begin{align*}
				(JR)^2 + I &= (JM^TSM)^2 + I \\
				&= (JM^TSM)(JM^TSM) + I \\
				&= JM^TS(MJM^T)SM + I\\
				&= JM^TSJSM + I 
			\end{align*} Since $M$ is symplectic and orthogonal, we have
			$MJM^T = J \Rightarrow JM^T = M^{-1}J = M^TJ.$ Hence we have 
			\begin{align*}
				(JR)^2 + I &= M^TJSJSM + I\\
				&= M^T(JS)^2M + I \\
				&= M^T(JS)^2M + M^TM \\
				&= M^T((JS)^2 + I)M.
			\end{align*} Since $(JS)^2 + I$ is of trace class and they form a two-sided ideal, we have the operator $(JR)^2 + I$ to be trace class.
		\end{enumerate}  Hence $R$ is a GCO when $S$ is a GCO. \\ The reverse implication can be obtained as follows. Since $R = M^TSM$, we have $S = (M^T)^{-1}RM^{-1} =  (M^{-1})^TRM^{-1}$. From Remark 3 \cite{john2018infinite}, $M^{-1}$ is a symplectic operator. Also since $M$ is orthogonal, so is $M^{-1}$. Put $L = M^{-1}$. Then $S = L^TRL$ for some symplectic and orthogonal operator $L$ on $\mathcal{H} \oplus \mathcal{H}$. Hence proceeding as above, we have $S$ is a GCO when $R$ is a GCO.
	\end{proof}
	
	Now we derive the conditions under which an operator in Class $\mathbb{B}$ becomes a GCO.
	
	\begin{theorem} \label{thmgcoinclassB}
		Let $S$ be an operator in Class $\mathbb{B}$. Then $S$ is a GCO if and only if
		\begin{enumerate}
			\item $\widehat{A+B} \geq \widehat{A-B}^{-1}$ on $\hat{\mathcal{H}}$,
			\item $A + B - I$ and $A - B - I$ are Hilbert-Schmidt operators on $\mathcal{H}$,
			\item $A^2 - B^2 - I$ is a trace class operator on $\mathcal{H}$.
		\end{enumerate}
	\end{theorem}
	
	\begin{proof}
		Given that the operator $S$ is in Class $\mathbb{B}$. Then $S$ takes the form 
		$
		S = \begin{bmatrix}
			A & B \\
			B & A
		\end{bmatrix},$ where $A$ and $B$ are self-adjoint operators on $\mathcal{H}$ such that the operators $A + B$ and $A - B$ are positive invertible and commuting. We have already seen that any $S$ in Class $\mathbb{B}$ is symplectically and orthogonally equivalent to $S^\prime = \begin{bmatrix}
			A + B & 0 \\
			0 & A - B
		\end{bmatrix}$. Hence from Lemma \ref{lemmaconditionongco}, $S$ is a GCO if and only if $S^\prime$ is a GCO. We have already derived conditions for the operator $S^\prime$ to be a GCO. Therefore we have the following. $S$ is a GCO if and only if
		\begin{enumerate}
			\item $\widehat{A + B} \geq (\widehat{A - B})^{-1}$ on $\hat{\mathcal{H}}$,
			\item The operators $A + B - I$ and $A - B - I$ are Hilbert-Schmidt operators on $\mathcal{H}$,
			\item $(A + B)(A - B) - I = A^2 - B^2 - I$ is a trace class operators on $\mathcal{H}$ (note that $(A + B)(A - B) = A^2 - B^2$ since $A$ and $B$ commutes).
		\end{enumerate}
	\end{proof}

	\subsection{Truncation method on GCOs} \label{subsectruncGCO}
	In this section, we show that the symplectic spectrum of a GCO can be completely retrieved using the truncation method described in \cite{bcn01}. First we derive some prerequisites needed to prove the result.
	\\ Let $\mathcal{M}$ be a complex separable Hilbert space, $A \in B(\mathcal{M})$ be a self-adjoint operator. Define  $$m = \underset{\|x \| = 1}{\textrm{ inf }} \langle Ax, x  \rangle; \quad \text{and} \quad M = \underset{\|x \| = 1}{\textrm{ sup }} \langle Ax, x  \rangle.$$ 
	Also let $\nu$ and $\mu$, respectively denote the minimum and maximum of the essential spectrum of $A$. The eigenvalues lying outside $[\nu, \mu]$ is at most countable, that is there are at most countable eigenvalues in the intervals $[m, \nu)$ and $(\mu, M]$ (See \cite{gohberg} for details). Therefore, the eigenvalues lying in the interval $[m, \nu)$ can be arranged in increasing order, and if they are countably infinite, they converge to $\nu$. Similarly, the eigenvalues in the interval $(\mu, M]$ can be arranged in decreasing order, and if they are countably infinite, they converge to $\mu$. Denote the eigenvalues in the interval $[m, \nu)$ as $$\lambda_1^-(A) \leq \lambda_2^-(A) \leq \cdots \leq \lambda_s^-(A)$$ and the eigenvalues in the interval $(\mu, M]$ as $$\lambda_R^+(A) \leq \cdots \leq \lambda_2^+(A) \leq \lambda_1^+(A)$$ (counting multiplicities) where $S,R \in \{1,2,\cdots \} \cup \{ \infty \}$. Since $\mathcal{M}$ is separable, it has a countable orthonormal basis, say, $\{e_1, e_2, \cdots \}$. For $n = 1,2, \cdots$, define $P_n$ to be the orthogonal projection of $\mathcal{M}$ onto the span of $\{ e_1, e_2, \cdots , e_n\}$. Put $A_n = P_nAP_n$, that is $A_n$ can be considered as the truncation of $A$ and can be identified with an $n \times n$ matrix. Denote the eigenvalues of $A_n$ (counting multiplicities) in non-increasing order by $$\lambda_1(A_n) \geq \lambda_2(A_n) \geq \cdots \geq \lambda_n(A_n).$$ A well-known result in the Galerkin method of spectral approximation \cite[Theorem 3.1]{bcn01} will serve as an important tool in our proof techniques. A formal statement of this result is provided below. 

	\begin{theorem}\cite[Theorem 3.1]{bcn01} \label{mnnresult}
		Let $A \in B(\mathcal{M})$ be self-adjoint. Then $$\underset{n \rightarrow \infty}{\lim} \lambda_k(A_n) = \begin{cases}
			{\lambda_k^+}(A), \textrm{ if } R = \infty, \textrm{ or } 1 \leq k \leq R \textrm{ when $ R <\infty$}, \\
			\mu, \textrm{ if } R < \infty \textrm{ and } k \geq R + 1,
		\end{cases}$$
		$$\underset{n \rightarrow \infty}{\lim} \lambda_{n+1-k}(A_n) = \begin{cases}
			{\lambda_k^-}(A),  \textrm{ if } S = \infty, \textrm{ or } 1 \leq k \leq S \textrm{ when $S < \infty$}, \\
			\nu, \textrm{ if } S < \infty \textrm{ and } k \geq S + 1,
		\end{cases}$$ In particular, $\underset{n \rightarrow \infty}{\lim} \lambda_k(A_n) \geq \mu \geq \nu \geq \underset{n \rightarrow \infty}{\lim} \lambda_{n+1-k}(A_n), \quad \forall k \geq 1,$ and 
		$\underset{k \rightarrow \infty}{\lim} \underset{n \rightarrow \infty}{\lim} \lambda_k(A_n) = \mu, \quad \underset{k \rightarrow \infty}{\lim} \underset{n \rightarrow \infty}{\lim} \lambda_{n+1-k}(A_n) = \nu.$
	\end{theorem}  
	
	The next theorem shows that the symplectic spectrum of a GCO can be completely retrieved using Theorem \ref{mnnresult}. Note that the spectrum of a GCO is atmost countable and its essential spectrum is $\{-1,1\}$ (see \cite{john2022interlacing} for details). Also from Lemma 3.3.6\cite{john2018infinite}, the symplectic spectrum of a GCO lies above $1$. 
	\begin{theorem} \label{thmgcomainresult}
		Let $S$ be a Gaussian Covariance Operator on $\mathcal{H} \oplus \mathcal{H}$. Denote the symplectic eigenvalues of $S$ in non-increasing order by  
		$d_1(S) \geq d_2(S) \geq \cdots \geq d_R(S)$ (counting multiplicities) where $R \in \{1,2, \cdots \} \cup \{ \infty\}$. Then
		\begin{equation*} \label{equationGCOtrunc}
			\underset{n \rightarrow \infty}{\lim} d_k(S_{2n}) = \begin{cases}
				{d_k}(S), \textrm{ if } R = \infty, \textrm{ or } 1 \leq k \leq R \textrm{ when $ R <\infty$}, \\
				1, \textrm{ if } R < \infty \textrm{ and } k \geq R + 1,
			\end{cases}
		\end{equation*} where 
		$d_1(S_{2n}) \geq d_2(S_{2n}) \geq \cdots \geq d_n(S_{2n})$ are the symplectic eigenvalues of $S_{2n}$, $n = 1,2, \cdots$ (arranged in non-increasing order).
	\end{theorem}
	
	\begin{proof}
		Let $\{e_1, e_2, \cdots \}$ be an orthonormal basis for $\mathcal{H}$. Then $\{(e_1,0), (e_2,0), \cdots \} \cup \{(0,e_1), (0,e_2), \cdots \}$ will be an orthonormal basis for $\mathcal{H} \oplus \mathcal{H}$. Denote the complexification of $\mathcal{H} \oplus \mathcal{H}$ by $\widehat{\mathcal{H} \oplus \mathcal{H}}$, that is $$\widehat{\mathcal{H} \oplus \mathcal{H}} = (\mathcal{H} \oplus \mathcal{H}) + i (\mathcal{H} \oplus \mathcal{H}).$$ Let $\{E_1, E_2, \cdots \}$ be an orthonormal basis for $\widehat{\mathcal{H} \oplus \mathcal{H}}$, where $$E_{2k-1}= (e_k,0) \textrm{ and } E_{2k} = (0,e_k),  \quad k = 1,2, \cdots.$$ For $n \in \mathbb{N}$, define the space $\widehat{(\mathcal{H} \oplus \mathcal{H})_{2n}}$ as follows.
		$$\widehat{(\mathcal{H} \oplus \mathcal{H})_{2n}} = \textrm{ span }\{E_1, E_2, \cdots, E_{2n}\}.$$ Let $P_{2n}$ denote the orthogonal projection of the Hilbert space $\widehat{\mathcal{H} \oplus \mathcal{H}}$ onto the subspace $\widehat{(\mathcal{H} \oplus \mathcal{H})_{2n}}$. Put $\widehat{i\sqrt{S}J\sqrt{S}}_{2n} = P_{2n}(\widehat{i\sqrt{S}J\sqrt{S}})P_{2n}$. Then $\widehat{i\sqrt{S}J\sqrt{S}}_{2n}$ can be identified with a matrix of order $2n$. Denote the eigenvalues of $\widehat{i\sqrt{S}J\sqrt{S}}_{2n}$ by 
		$$\lambda_1(\widehat{i\sqrt{S}J\sqrt{S}}_{2n}) \geq \lambda_2(\widehat{i\sqrt{S}J\sqrt{S}}_{2n}) \geq \cdots \geq \lambda_{2n}(\widehat{i\sqrt{S}J\sqrt{S}}_{2n}).$$ 		
		Consider the eigenvalues of $\widehat{i\sqrt{S}J\sqrt{S}}$ to the right of $1$ denoted by 
		$$\lambda_R^+(\widehat{i\sqrt{S}J\sqrt{S}}) \leq \cdots \leq \lambda_2^+(\widehat{i\sqrt{S}J\sqrt{S}}) \leq \lambda_1^+(\widehat{i\sqrt{S}J\sqrt{S}})$$
		(counting multiplicities) where $R \in \{1,2, \cdots\} \cup \{\infty\}$. Now from Theorem \ref{mnnresult}, we have 
		$$\underset{n \rightarrow \infty}{\lim} \lambda_k(\widehat{i\sqrt{S}J\sqrt{S}}_{2n}) = \begin{cases}
			{\lambda_k^+}(\widehat{i\sqrt{S}J\sqrt{S}}), \textrm{ if } R = \infty, \textrm{ or } 1 \leq k \leq R \textrm{ when $ R <\infty$}, \\
			1, \textrm{ if } R < \infty \textrm{ and } k \geq R + 1,
		\end{cases}$$
		But the spectrum of an operator and its complexification are the same. Hence the above relation implies
		\begin{equation} \label{equationtruncation}
			\underset{n \rightarrow \infty}{\lim} \lambda_k(i\sqrt{S}J\sqrt{S}_{2n}) = \begin{cases}
				{\lambda_k^+}(i\sqrt{S}J\sqrt{S}), \textrm{ if } R = \infty, \textrm{ or } 1 \leq k \leq R \textrm{ when $ R <\infty$}, \\
				1, \textrm{ if } R < \infty \textrm{ and } k \geq R + 1,
			\end{cases}
		\end{equation}
		
		It is known that  the symplectic eigenvalues of an operator $Q$ are the positive eigenvalues of the self-adjoint operator $i\sqrt{Q}J\sqrt{Q}$ (see  \cite[Corollary 3.3 and Theorem 4.3]{bv01}). The same holds for matrices as well (see \cite{rbt01}). Also from Lemma 3.3.6\cite{john2018infinite}, the symplectic spectrum of a GCO lies above $1$. Hence Equation \ref{equationtruncation} actually exhibits a method to completely retrieve the symplectic eigenvalues of $S$ using its finite dimensional truncations and we have 
		$$
		\underset{n \rightarrow \infty}{\lim} d_k(S_{2n}) = \begin{cases}
			{d_k}(S), \textrm{ if } R = \infty, \textrm{ or } 1 \leq k \leq R \textrm{ when $ R <\infty$}, \\
			1, \textrm{ if } R < \infty \textrm{ and } k \geq R + 1.
		\end{cases}
		$$
	\end{proof}

	\section{Numerical Illustrations} \label{secnumillus}
	
	In this section, we will illustrate numerical examples where the truncation method proves effective. All computations are performed using the Python program we have developed.
	
	\begin{example} \label{exptrunwins}
		Let $\mathcal{H} = l^2$. Define $A : l^2 \rightarrow l^2$ to be the Toeplitz operators corresponding to the symbol $2 + cos(t) \in L^\infty(\mathbb{T})$. Then $A$ takes the form 
		$$ A = 
		\begin{bmatrix}
			2 & \frac{1}{2} & 0 & 0 & \cdots \\
			\frac{1}{2} & 2 & \frac{1}{2} & 0 & \cdots \\
			0 & \frac{1}{2} & 2 & \frac{1}{2} & \cdots \\
			\cdots & \cdots & \cdots & \cdots & \cdots 
		\end{bmatrix}
		$$ Then $\sigma(A) = [1,3]$. Define $T : l^2 \oplus l^2 \rightarrow l^2 \oplus l^2$ as $T = \begin{bmatrix}
			A & 0 \\ 
			0 & A
		\end{bmatrix}$. Since $A$ is a Toeplitz operator, its essential spectrum coincides with the spectrum, that is $\sigma_{ess}(A)$ is connected. Hence from Theorem \ref{rstAA}, the method of truncation will work. We will now verify this using the algorithm we have developed (the notations are as in the proof of Theorem \ref{rstAA}). The table below gives the symplectic eigenvalues of the truncated operator $T_{2n}$ of $T$.

		\begin{table}[htb] \label{numericalexample1}
			\begin{tabular}{|p{0.05\linewidth} | p{0.9\linewidth}|} 
				\hline
				$n$  & Symplectic Eigenvalues (upto 5 decimal places) \\
				\hline
				10 & {1.13397, 1.5, 2.0, 2.5, 2.86603} \\
				\hline
				20 & 1.04051, 1.15875, 1.34514, 1.58458, 1.85769, 2.14231, 2.41542, 2.65486, 2.84125, 2.95949 \\
				\hline 
				50 & 1.00729, 1.02906, 1.06498, 1.11454, 1.17702, 1.25149, 1.33688, 1.43194, 1.53528, 1.6454, 1.76068, 1.87946, 2.0, 2.12054, 2.23932, 2.3546, 2.46472, 2.56806, 2.66312, 2.74851, 2.82298, 2.88546, 2.93502, 2.97094, 2.99271 \\
				\hline 
				100 & 1.0019, 1.00758, 1.01703, 1.0302, 1.04706, 1.06753, 1.09153, 1.11899, 1.14978, 1.1838, 1.22092, 1.26099, $\cdots$ , 1.66765, 1.72634, 1.78607, 1.84661, 1.90773, 1.9692, 2.0308, 2.09227, 2.15339, 2.21393, 2.27366, 2.33235, $\cdots$, 2.77908, 2.8162, 2.85022, 2.88101, 2.90847, 2.93247, 2.95294, 2.9698, 2.98297, 2.99242, 2.9981 \\ \hline
				$\cdots$ & $\cdots$ \\
				\hline 
				1000 & 1.00002, 1.00008, 1.00018, 1.00031, 1.00049, 1.00071, 1.00096, 1.00126, 1.00159, 1.00197, 1.00238, 1.00283, $\cdots$, 1.96552, 1.97179, 1.97805, 1.98432, 1.99059, 1.99686, 2.00314, 2.00941, 2.01568, 2.02195, 2.02821, 2.03448, $\cdots$, 2.99668, 2.99717, 2.99762, 2.99803, 2.99841, 2.99874, 2.99904, 2.99929, 2.99951, 2.99969, 2.99982, 2.99992, 2.99998  \\
				\hline 
				$\cdots$ & $\cdots$ \\
				\hline
			\end{tabular}
			\caption{Truncated Symplectic Eigenvalues for $T$ in Example \ref{exptrunwins}.}  
		\end{table}

		\noindent It can be observed from Table \ref{numericalexample1} that as $n$ increases, the symplectic eigenvalues tend to take more values in the interval $[1,3]$, which is the symplectic spectrum of $T$ by Theorem \ref{rstAA}.

	\end{example}
	
	\begin{example} {(GCO)}\label{numericalexgco}
		Consider the operator $T$ defined on $\mathcal{H} \oplus \mathcal{H}$ by
		$$
		T = \begin{bmatrix}
			A & B \\
			B & A
		\end{bmatrix},$$ where $A$ and $B$ are operators on $\mathcal{H}$ given by
		\begin{align*}
			A &= \textrm{ diag } \left\{ 1 + \frac{1}{2(n+1)^2} + \frac{1}{2(n+1)^3} : n \in \mathbb{N} \right\} \\
			B &= \textrm{ diag } \left\{ \frac{1}{2(n+1)^2} - \frac{1}{2(n+1)^3} : n \in \mathbb{N} \right\}. \\
		\end{align*}
		From Theorem \ref{thmgcoinclassB}, it can be verified that the operator $T$ is a GCO that comes in class $\mathbb{B}$. Then following the notations as in Theorem \ref{rstABBA}, the truncation of the operator $T$ is given by
		$$
		T_{2n} = \begin{bmatrix}
			A_n & B_n \\
			B_n & A_n
		\end{bmatrix}.
		$$
		The symplectic eigenvalues of the truncated matrix $T_{2n}$ is given in Table \ref{numericalexgco}.
		
		\newpage
		\begin{table}[htb]
			\begin{tabular}{|p{0.05\linewidth} | p{0.9\linewidth}|} 
				\hline
				$n$  & Symplectic Eigenvalues (upto 10 decimal places)\\
				\hline
				10 & 1.1858541226, 1.0734353145, 1.0387981337, 1.0238749924, 1.01613779 \\
				\hline
				20 &  1.1858541226, 1.0734353145, 1.0387981337, 1.0238749924, 1.01613779, 1.0116239984, 1.0087659008, 1.0068437596, 1.0054899303, 1.0045008645 \\
				\hline 
				50 & 1.1858541226, 1.0734353145, 1.0387981337, 1.0238749924, 1.01613779, 1.0116239984, 1.0087659008, 1.0068437596, 1.0054899303, 1.0045008645, 1.0037565277, 1.0031824456, 1.0027304382, 1.0023682246, 1.0020735224, 1.0018305513, 1.0016278834, 1.0014570789, 1.0013117958, 1.0011871944, 1.0010795294, 1.0009858661, 1.0009038788, 1.0008317053, 1.0007678401  \\
				\hline 
				100 & 1.1858541226, 1.0734353145, 1.0387981337, 1.0238749924, 1.01613779, 1.0116239984, 1.0087659008, 1.0068437596, 1.0054899303, 1.0045008645, 1.0037565277, 1.0031824456, 1.0027304382, 1.0023682246, 1.0020735224, 1.0018305513, 1.0016278834, 1.0014570789, 1.0013117958, 1.0011871944, 1.0010795294, 1.0009858661, 1.0009038788, 1.0008317053, 1.0007678401, 1.0007110557, 1.0006603431, 1.0006148667, 1.00057393, 1.0005369483, 1.0005034282, 1.000472951, 1.0004451592, 1.0004197465, 1.0003964489, 1.0003750381, 1.0003553157, 1.0003371088, 1.0003202661, 1.0003046546, 1.0002901572, 1.0002766703, 1.0002641023, 1.0002523714, 1.000241405, 1.0002311381, 1.0002215124, 1.0002124757, 1.0002039808, 1.0001959853 \\
				\hline
				$\cdots$ & $\cdots$ \\
				\hline 
				1000 & 1.1858541226, 1.0734353145, 1.0387981337, 1.0238749924, 1.01613779, 1.0116239984, 1.0087659008, 1.0068437596, 1.0054899303, 1.0045008645, 1.0037565277, 1.0031824456, 1.0027304382, 1.0023682246, 1.0020735224, 1.0018305513, 1.0016278834, 1.0014570789, 1.0013117958, 1.0011871944, 1.0010795294, 1.0009858661, 1.0009038788, 1.0008317053, 1.0007678401, 1.0007110557, 1.0006603431, 1.0006148667, 1.00057393, $\cdots$, 
				1.0000024205, 1.0000024099, 1.0000023993, 1.0000023888, 1.0000023784, 1.0000023681, 1.0000023578, 1.0000023476, 1.0000023375, 1.0000023274, 1.0000023174, 1.0000023074, 1.0000022975, 1.0000022877, 1.000002278, 1.0000022683, 1.0000022586, 1.0000022491, 1.0000022396, 1.0000022301, 1.0000022207, 1.0000022114, 1.0000022021, 1.0000021929, 1.0000021838, 1.0000021747, 1.0000021656, 1.0000021566, 1.0000021477, 1.0000021388, 1.00000213, 1.0000021212, 1.0000021125, 1.0000021039, 1.0000020953, 1.0000020867, 1.0000020782, 1.0000020698, 1.0000020614, 1.000002053, 1.0000020447, 1.0000020365, 1.0000020283, 1.0000020201, 1.000002012, 1.000002004, 1.000001996  \\
				\hline 
				$\cdots$ & $\cdots$ \\
				\hline
			\end{tabular}
			\caption{Truncated Symplectic Eigenvalues for $T$ in Example \ref{numericalexgco}.}  
		\end{table} 
		
		From Theorem \ref{thmgcomainresult}, we have $d_k(T) = \underset{n \rightarrow \infty}{\lim} d_k(T_{2n})$. Now by observing the values given in Table \ref{numericalexgco}, it can be seen that for each $k = 1,2, \cdots, n$ where $n \in \mathbb{N}$, $d_k(T_{2n})$ takes the form
		$
		\sqrt{\left( 1 + \frac{1}{(k+1)^2} \right)\left( 1 + \frac{1}{(k+1)^3} \right)}.$ 
		Hence from the numerical computations we have,
		$d_k(T) = \underset{n \rightarrow \infty}{\lim} d_k(T_{2n}) = \sqrt{\left( 1 + \frac{1}{(k+1)^2} \right)\left( 1 + \frac{1}{(k+1)^3} \right)}, \,\, k = 1,2, \cdots.$
		Now let us verify this numerical observation using Theorem \ref{rstABBA}. The given operator $T$ is symplectically equivalent to the operator 
		$$
		T^\prime = \begin{bmatrix}
			A + B & 0 \\
			0 & A - B
		\end{bmatrix},
		$$ where 
		$$
		A + B = \textrm{ diag }\left\{ 1 + \frac{1}{(n+1)^2}: n \in \mathbb{N} \right\},
		$$
		$$
		A - B = \textrm{ diag }\left\{ 1 + \frac{1}{(n+1)^3}: n \in \mathbb{N} \right\}.
		$$ Then from Theorem \ref{rstAB}, 
		$\sigma_{sy}(T) = \sigma((A + B)^{\frac{1}{2}}(A - B)^{\frac{1}{2}}) = \left\{ \sqrt{\left( 1 + \frac{1}{(n+1)^2} \right)\left( 1 + \frac{1}{(n+1)^3} \right)}: n \in \mathbb{N} \right\}.$
		Hence the theoretical value coincides with the values obtained through numerical computations.
	\end{example}

	\section{Concluding Remarks and Future Problems}
	
	Here we list the major achievements of this article and some future problems. One of our accomplishments is successfully determining the Williamson's normal form and establishing bounds for the symplectic spectrum for operators within a specific class. Additionally, we identified conditions under which the considered operators become a Gaussian Covariance Operator (GCO). Notably, we demonstrated that the symplectic spectrum of a GCO can be entirely obtained using its finite-dimensional truncations. To exemplify the truncation method, we developed an illustrative algorithm.
	
	The following are some of the future problems we intend to explore:
	\begin{enumerate}
		\item Firstly, we aim to investigate the Williamson's normal form and symplectic spectrum of integral operators and infinite matrices.
		
		\item We also wish to explore scenarios in which integral operators and infinite matrices become a Gaussian Covariance Operator.
	\end{enumerate}
	Through addressing these future problems, we seek to further expand our understanding of symplectic spectrum and their implications in both finite and infinite-dimensional contexts.
	
	\section*{Acknowledgements}
	V. B. Kiran Kumar thanks the SERB SURE Scheme for the financial support. Anmary Tonny is supported by the INSPIRE PhD Fellowship of the Department of Science and Technology, Govt of India. The authors wish to thank Dr. Tiju Cherian John, Research Scientist, University of Arizona for the fruitful discussions. 

	\nocite{*}
	\bibliography{jmp}
	
\end{document}